\documentclass{amsart}

\usepackage{amsmath}
\usepackage{amssymb}
\usepackage{amsthm}
\usepackage{amscd}
\usepackage{amsfonts}
\usepackage{graphicx}%
\usepackage{fancyhdr, geometry}
\usepackage[usenames,dvipsnames]{xcolor}
\usepackage{url, hyperref, verbatim}%, ulem}
\usepackage{mathtools}
\usepackage{tikz}
\usetikzlibrary{matrix}
\usepackage{tabularx}
\usepackage{booktabs}
\usepackage{enumerate}

\theoremstyle{plain} %\numberwithin{equation}{section}
\newtheorem{thm}{Theorem}[section]

\newtheorem{cor}[thm]{Corollary}

\newtheorem{lm}[thm]{Lemma}

\newtheorem{prop}[thm]{Proposition}
\theoremstyle{definition}
\newtheorem{defn}[thm]{Definition}
\newtheorem{conv}[thm]{Convention}
\newtheorem{notation}[thm]{Notation}

\newtheorem{remark}[thm]{Remark}

\newtheorem{example}[thm]{Example}

\usepackage{tikz}
\usepackage{stmaryrd}

\newcommand{\bigoh}[1]{\mathrm{O}\left(#1\right)}
\newcommand{\indij}{1_{\{i,j\}}}
\newcommand{\indijprime}{1_{\{i',j'\}}}
\newcommand{\var}{\text{Var}}
\newcommand{\cov}{\text{Cov}}
\newcommand{\expect}[1]{\mathbb{E}\left[#1\right]}
\newcommand{\prob}[1]{\mathbb{P}\left[#1\right]}
\newcommand{\nsg}{\mathcal{S}}
\newcommand{\gens}{\mathcal{A}}

\newcommand{\ring}[1]{\ensuremath{\mathbb{#1}}}

\renewcommand\>{\rangle}

\newcommand\ZZ{\ring{Z}}

\newcommand\kk{\Bbbk}

\newcommand{\excise}[1]{}%{$\star$\textsc{#1}$\star$}

\begin{document}
\title{Random numerical semigroups and a simplicial complex of irreducible semigroups}

\author{J. A. De Loera, C. O'Neill, D. Wilburne}
\date{\today}

\begin{abstract}
\hspace{-2.05032pt}
We examine properties of random numerical semigroups  under a probabilistic model inspired by the Erd\"os-R\'enyi model for random graphs. We provide a threshold function for cofiniteness, and bound the expected embedding dimension, genus, and Frobenius number of random semigroups.  Our results follow, surprisingly, from the construction of a very natural shellable simplicial complex whose facets are in bijection with irreducible numerical semigroups of a fixed Frobenius number and whose $h$-vector determines the probability that a particular element lies in the semigroup.  
\end{abstract}

\maketitle

%!TEX root = 0-RandomSemigroups.tex

%%%%%%%%%%%%%%%%%%%%%%%%%%%%%%%%%%%%%%%%%%%%%%%%%%%%%%%%%%%%
\section{Introduction}%%%%%%%%%%%%%%%%%%%%%%%%%%%%%%%%%%%%%%
\label{sec:intro}%%%%%%%%%%%%%%%%%%%%%%%%%%%%%%%%%%%%%%%%%%%
%%%%%%%%%%%%%%%%%%%%%%%%%%%%%%%%%%%%%%%%%%%%%%%%%%%%%%%%%%%%

A \emph{numerical semigroup} is a subset $S$ of the non-negative integers $\mathbb Z_{\ge 0}$ that is closed under addition.  A~nonnegative integer $n$ is a \emph{gap} of $S$ if $n \notin S$ and we denote the set of gaps of $S$ by $G(S)$.  Numerical semigroups appear in several areas of mathematics~\cite{assi+pedro}, and there are several interesting combinatorial invariants of a semigroup~\cite{rosales+pedro}.  Notable numerical semigroup invariants include the \emph{embedding dimension} $e(S)$, which is the number of minimal generators of $S$ and the \emph{genus} $g(S)$, which is the number of gaps of $S$, i.e.,\ $g(S) = \#G(S)$, and the \emph{Frobenius number} $F(S)$, which is the largest gap of $S$.  The latter two invariants are usually only defined when $S$ is \emph{cofinite}, that is, when $S$ has finite complement in $\mathbb Z_{\ge 0}$.  The theory of numerical semigroups is a vibrant subject, with connections to algebraic geometry and commutative algebra \cite{abhyankar1967local, barucci1997maximality, brown1992one, campillo1983onsaturations, kunz1970value, sally1977associated} as well as integer optimization and number theory (see \cite{assi+pedro} and references therein).  In this paper, we investigate invariants of numerical semigroups from a probabilistic point of view.

V.~Arnol'd \cite{arnoldo}, J.~Bourgain and Y.~Sinai \cite{bourgain+sinai} initiated the study of  the ``average behavior''  of numerical semigroups by analyzing the Frobenius function $F(S)$ for ``typical'' numerical semigroups (see more recent work in \cite{alievhenkaicke}).  Each of these papers produced random numerical semigroups using the uniform probability distribution on the collection
$$G(N,T) = \{\mathbf a \in \mathbb Z^N_{>0} \colon \gcd(\mathbf a) = 1 \text{ and } \|\mathbf a\| \leq T\}$$
of generating sets, and each proved several interesting statements about the expected value (in the usual probabilistic sense) of the Frobenius number.  See the references in \cite{alievhenkaicke} for a thorough overview.  

In this paper, we study a different model, which generates at random a numerical semigroup $\nsg$ according to the following procedure: 

\begin{enumerate}
\item fix a nonnegative integer $M$ and a probability $p \in [0,1]$;
\item initialize a set of generators $\gens=\{0\}$ for $\nsg$;
\item independently choose with probability $p$ whether to include each $n \le M$ in $\gens$.
\end{enumerate}
The notation $\nsg \sim S(M,p)$ indicates $\nsg$ is a random numerical semigroup produced with this model.  
A similar model was recently used to produce random monomial ideals \cite{rmi} (that is, each multivariate monomial with bounded total degree is included in a generating set with probability~$p$).  The authors dubbed this model the ``ER-type model'' for its resemblance to the Erd\H os-R\'enyi  model of random graphs~\cite{Erdos+Renyi}; we will use the same convention.  

Unlike previously used models, which sampled uniformly among numerical semigroups with a fixed number of generators, the ER-type model allows one to specify a probability as input, yielding more refined control over the numerical semigroups produced.  Our model is also more closely aligned with the ``standard'' sampling methods from probabilistic combinatorics, and more compatible with the use of numerical semigroups in integer programming, where non-cofinite semigroups occur alongside cofinite ones.  

Our main result is as follows.  

\begin{thm}\label{t:maintheorem}
Let $\nsg\sim S(M,p)$, where $p = p(M)$ is a monotone decreasing function of $M$.
\begin{enumerate}[a)]
\item If $p \ll \frac{1}{M},$ then $\nsg = \langle 0\rangle$ a.a.s.
\item If $\frac{1}{M} \ll p \ll 1$, then $\nsg$ is cofinite, i.e.,  the set of gaps is finite, a.a.s and 
$$\lim_{M \to \infty} \expect{e(\nsg)} = \lim_{M \to \infty} \expect{g(\nsg)} = \lim_{M \to \infty} \expect{F(\nsg)} = \infty.$$
\item If $\lim_{M \to \infty} p(M) > 0$, then 
$$\lim_{M \to \infty} \expect{e(\nsg)}, \lim_{M \to \infty} \expect{g(\nsg)}, \lim_{M \to \infty} \expect{F(\nsg)} < \infty,$$
and each limit is bounded by explicit rational functions in $p$ (see page~11).  
\end{enumerate}
\end{thm}

%\chris{Perhaps we should combine the above Thm and Cor into one main result, stating the thresholds for both cofiniteness ($p = 1/M$) and finite expected mingens ($p \gg 0$).  
%Different sections can then prove different parts, and at the end of the last section we have a one line ``proof of main Corollary''.}

Although part~(a) of Theorem~\ref{t:maintheorem} follows from standard arguments in probabilistic combinatorics (Theorem~\ref{t:cofinite}), parts~(b) and~(c) follow, surprisingly, from the construction of a very natural shellable simplicial complex (Definition~\ref{d:complexofirreducibles}) whose facets are in bijection with irreducible numerical semigroups of a fixed Frobenius number (Definition~\ref{d:irreducible}).  As it turns out, some of the probabilities involved in determining the expected values above require precisely the $h$-vector (in the sense of algebraic combinatorics \cite{stanleybook}) for this simplicial complex.  Through the $h$-vector, we distinguish parts~(b) and~(c) of Theorem~\ref{t:maintheorem} (Corollary~\ref{c:embdim}) and estimate the finite expectations (Theorem~\ref{t:fixedpbounds}).  

\subsection*{Acknowledgements}

The authors would like to thank Iskander Aliev,  Martin Henk, Calvin Leng, and Pedro Garc{\'\i}a-Sanchez for several helpful conversations and suggestions, and are grateful to Zachary Spaulding for assisting with the experiments in Table~\ref{tb:fixedpbounds}.  The first and second author were partially supported by NSF grant DMS-1522158 to the University of California\ Davis. The first author was also partially supported by NSF grant DMS-1440140, while he visited the Mathematical Sciences Research Institute in Berkeley, California, during the Fall 2017 semester.  The third author was supported by NSF collaborative grant DMS-1522662 to Illinois Institute of Technology.

%!TEX root = 0-RandomSemigroups.tex

%%%%%%%%%%%%%%%%%%%%%%%%%%%%%%%%%%%%%%%%%%%%%%%%%%%%%%%%%%%%
\section{Background}%%%%%%%%%%%%%%%%%%%%%%%%%%%%%%%%%%%%%%%%
\label{sec:background}%%%%%%%%%%%%%%%%%%%%%%%%%%%%%%%%%%%%%%
%%%%%%%%%%%%%%%%%%%%%%%%%%%%%%%%%%%%%%%%%%%%%%%%%%%%%%%%%%%%

% For details on the probabilistic method we recommend, e.g.,\ \cite{AlonSpencer, bollobasbook}.

We begin by recalling some basic notions from probability theory and establishing notation used throughout the paper.  For a more comprehensive resource on methods in probabilistic combinatorics, we refer the reader the excellent book of Alon and Spencer \cite{AlonSpencer}.  The \emph{expected value} of a discrete random variable $X$ taking values in $\Omega\subseteq\mathbb{N}$ is $\expect{X}=\sum_{n\in\Omega}n\cdot \prob{X=n}$, and the \emph{variance} of $X$ is given by $\var{[X]}=\expect{X^2}-\expect{X}^2$.  The~most useful property of the expectation operator is its linearity: $\expect{aX+bY}=a\expect{X}+b\expect{Y}$ for random variables $X,Y$ and scalars $a,b$. The \emph{covariance} of two random variables $X,Y$ is given by $\cov{[X,Y]}=\expect{XY}-\expect{X}\expect{Y}$.  The variance of a sum of random variables $X=X_1+\ldots+X_n$ can be computed via the formula 
\[
\var{[X]}=\sum_{i=1}^n\var{[X]}+\sum_{i<j}\cov{[X_i,X_j]}.
\]  The \emph{indicator random variable} $1_E$ of an event $E$ takes the value 1 if $E$ occurs and is 0 otherwise.  If~$E$ occurs with probability $p$, then $\expect{1_E}=p$ and $\var{[1_E]}=p(1-p)$.  If $E$ depends on a parameter~$n$ and $\lim_{n\to\infty}\prob{E}=1$, we say that $E$ holds \emph{asymptotically almost surely} (abbreviated \emph{a.a.s.}).

In Theorem~\ref{t:cofinite}, we will bound the probability that a nonnegative integer random variable~$X$ is non-zero using the so-called \emph{moment method} techniques.  The \emph{first moment method}, a consequence of Markov's inequality~\cite{AlonSpencer}, says that the probability $X$ is non-zero is bounded above by its expectation, i.e. $\prob{X\not=0}\le\expect{X}.$   On the other hand, the \emph{second moment method} provides an upper bound for the probability that $X$ is zero in terms of its variance, $\prob{X=0}\le \var{[X]}/\expect{X}^2$, and follows from Chebyshev's inequality \cite[Theorem 4.1.1]{AlonSpencer}.

For functions $f,g$ depending on some parameter $n$, we use the notation $f\gg g$ to indicate that the ratio $g/f\to 0$ as $n\to\infty$ and similarly $f\ll g$ means that $f/g\to 0$ as $n\to\infty$.  With this, clearly $f\ll 1$ implies $f\to 0$ as $n\to\infty$.

%\dane{in the above paragraph, do we need that $f$ and $g$ are monotone decreasing for the last sentence to be true?}

In the theory of Erd\H os-R\'enyi random graphs, many graph properties tend to appear or not appear with high probability based on the asymptotics of the probability parameter $p$.  This phenomenon is quantified by the notion of a \emph{threshold function} (see \cite[Chapter 10]{AlonSpencer}).  Here we define a notion of threshold function tailored to the context of random numerical semigroups.  We say that a property $\mathcal{P}$ of a numerical semigroup $S$ is \emph{monotone} if $S$ has $\mathcal{P}$ and if $S\cup\{s\}$ for $s\in\mathbb{N}$ is still a numerical semigroup, then $S\cup\{s\}$ also has $\mathcal{P}$.   For example, the property of being cofinite is monotone, whereas the property of having Frobenius number $n$ is not.  A \emph{threshold function} for a monotone numerical semigroup property $\mathcal{P}$ is a function $t(M)$ such that
\[
\prob{\mathcal{S}\text{ has }\mathcal{P}}=
\begin{cases}
1,&\text{if }p(M)\gg t(M)\\
0,&\text{if } p(M)\ll t(M)\\
\end{cases}
\] for $\nsg\sim S(M,p)$.  Loosely speaking, if $t$ is a threshold for a property $\mathcal{P}$, when $p$ is much smaller than $t$, then $\nsg\sim S(M,p)$ will not have $\mathcal{P}$ a.a.s., and if $p$ is much larger than $t$, $\nsg$ will have $\mathcal{P}$ a.a.s.

%\dane{fill in probability stuff (just what is needed for our results, i.e.\ threshold functions, moment methods, Hoeffding's inequality, etc), if it turns out to be sufficiently short perhaps we can simply intersperse these throughout the document as they are needed}

%!TEX root = 0-RandomSemigroups.tex

%%%%%%%%%%%%%%%%%%%%%%%%%%%%%%%%%%%%%%%%%%%%%%%%%%%%%%%%%%%%
\section{Distribution and cofiniteness}%%%%%%%%%%%%%%%%%%%%%
\label{sec:cofinite}%%%%%%%%%%%%%%%%%%%%%%%%%%%%%%%%%%%%%%%%
%%%%%%%%%%%%%%%%%%%%%%%%%%%%%%%%%%%%%%%%%%%%%%%%%%%%%%%%%%%%

%The following lemma will be useful throughout.

%\begin{lm}\label{l:booleansum} Let $\mathcal{B}_n$ denote the boolean lattice on $[n]$. Then,
%\[
%\sum_{B\in\mathcal{B}_n}t^{|B|}(1-t)^{n-|B|}=1.
%\]
%\end{lm}

%\begin{proof}  The claim follows easily from the binomial theorem, since
%\[
%\sum_{B\in\mathcal{B}_n}t^{|B|}(1-t)^{n-|B|}=\sum_{k=0}^n\binom{n}{k}t^k(1-t)^{n-k}=1.
%\]
%\end{proof}

In this section, we prove that the threshold function for cofiniteness coincides with the threshold function for nonemptyness (Theorem~\ref{t:cofinite}).  First, we give Theorem~\ref{t:distribution}, which states the probability of observing a fixed numerical semigroup in terms of its embedding dimension and gaps.  

\begin{notation}\label{n:gapsbelowm}
In what follows, we denote by $g_M(S)$ the number of gaps $n$ of $S$ such that $n \le M$.  
\end{notation}

\begin{thm}\label{t:distribution}
Let $\nsg\sim S(M,p)$ and let $S$ be a numerical semigroup with minimal generating set $A \subset [M]$.  Then, 
\[
	\prob{\nsg=S}= p^{e(S)}(1-p)^{g_M(S)}.
\]  
If, in addition, $F(S)\le M$, then
\[
	\prob{\nsg=S}= p^{e(S)}(1-p)^{g(S)}.
\]  
\label{thm:dist}
\end{thm}

\begin{proof}
First, observe that $\nsg = S$ if and only if $\gens \supset A$ and no gap $g$ of $\nsg$ is in $\gens$. Let $\mathfrak{A}$ consist of the elements of $[M]$ that are neither minimal generators nor gaps of $S$, i.e. $\mathfrak{A}:=(S\setminus A)\cap[M]$.  %Define $L:=(\mathfrak{A},\subset).$  It is easily verified that  $L$ is isomorphic (as a lattice) to the Boolean lattice $\mathcal{B}_{n}$, where $n=M-e(S)-g_M(S)$.  
Now, for any $B\subset [M]$ such that $B\supset A$ and $B\cap G(S)=\emptyset$, one may write $B$ as a disjoint union $B=A\cup C$, where $C\in\mathfrak{A}$.  Since $e(S)=|A|$,
\[
\prob{\gens = B}=p^{|B|}(1-p)^{M-|B|}=p^{e(S)}p^{|C|}(1-p)^{g_M(S)}(1-p)^{M-e(S)-g_M(S)-|C|}.
\] Hence, 
\[
\prob{\nsg=S}=\sum_{\substack{B\supset A\\ B\cap G(S)=\emptyset}}\prob{\gens = B}=p^{e(S)}(1-p)^{g_M(S)}\sum_{C\in\mathfrak{A}}p^{|C|}(1-p)^{M-e(S)-g_M(S)-|C|}.
\]  By the binomial theorem, $\sum_{C\in\mathfrak{A}}p^{|C|}(1-p)^{M-e(S)-g_M(S)-|C|}=1,$ which completes the proof of the first claim.  The second claim follows since if $F(S)\le M$, then $g_M(S)=g(S)$.
%%%%%%%%%%%%%%%OLD PROOF%%%%%%%%%%%%%%%%%%%%%%%%%%
%Let $\mathfrak{A}\subset2^{[M]}$ denote the collection of subsets satisfying these two conditions:
%\[
%\mathfrak{A}:=\{B\subset[M]:B\supset A\text{ and } B\cap G(S)=\emptyset\}.
%\]  For any $B\in\mathfrak{A},$ $P(B)=p^{|B|}(1-p)^{M-|B|}=p^{e(S)}p^{|B\setminus A|}(1-p)^{g_M(S)}(1-p)^{|(S\cap [M])\setminus B|}.$  Thus,
%\[
%P(\nsg=S)=\sum_{B\in\mathfrak{A}}P(B)=p^{e(S)}(1-p)^{g_M(S)}\sum_{B\in\mathfrak{A}}p^{|B\setminus A|}(1-p)^{|(S\cap [M])\setminus B|}.
%\]  Since we must have $e(s)+g_M(S)+|B\setminus A|+|(S\cap [M])\setminus B|=M$, it follows that
%\[
%\sum_{B\in\mathfrak{A}}p^{|B\setminus A|}(1-p)^{|(S\cap [M])\setminus B|}=\sum_{i=0}^{M-e(S)-g_M(S)}\binom{M-e(S)-g_M(S)}{i}p^i(1-p)^{M-e(S)-g_M(S)-i}=1,
%\] which shows that indeed $P(\nsg=S)=p^{e(S)}(1-p)^{g_M(S)}$.  
\end{proof}

\begin{lm}\label{l:zero}
Let $\nsg\sim S(M,p)$.  If $p\ll 1/M,$ then $\nsg=\langle0\rangle$ a.a.s.
\label{thm:zero}
\end{lm}

\begin{proof} The condition $p\ll 1/M$ implies $pM\to0$ as $M\to\infty$.  Thus, \mbox{$\mathbb{E}[|\gens|]=pM\to0\text{ as } M\to\infty$} and the result follows by applying Markov's inequality.
\end{proof}

The proof of the threshold function for cofiniteness relies on counting the number of coprime pairs in the random set $\gens$.  

\begin{thm}\label{t:cofinite}
The function $t(M)=1/M$ is a threshold function for the property that $\nsg\sim S(M,p)$ is cofinite.
\end{thm}

\begin{proof} 
By Lemma~\ref{l:zero}, when $p\ll 1/M,$ $\nsg$ is the non-cofinite numerical semigroup $		\langle 0\rangle$ a.a.s.  Consider the case when $p\gg 1/M.$  For each pair $\{i,j\}$ such that $1\le i<j\le M,$ define the indicator random variable $\indij$ as follows:
	
\[
	\indij:=
	\begin{cases}
		1,&\text{if } i,j\in \gens \\
		0,&\text{otherwise}.\\
	\end{cases}
\]
Let $X$ be the the number of pairs of coprime integers in $\gens$, so that
\[
	X=\sum_{\substack{ i<j \\ (i,j)=1}}\indij.
\]
 By the basic properties of indicator random variables,  $\mathbb{E}[\indij]=p^2 \text{ and } \var[\indij]=p^2(1-p^2).$  If $\{i,j\}\cap\{i',j'\}=\emptyset$, $\cov[\indij,\indijprime]=0,$ since the events $\indij=1$ and $\indijprime=1$ are independent.  If $i=i'$ and $j\not=j',$ 
\[
	\cov[\indij,1_{\{i,j'\}}]=\mathbb{E}[\indij\cdot1_{\{i,j'\}}]-\mathbb{E}[\indij]\mathbb{E}[1_{\{i,j'\}}]=p^3(1-p).
\]
Thus,
\[
	\mathbb{E}[X]=\sum_{\substack{i<j\\ (i,j)=1}}p^2\le M^2p^2,%=\sum_{n=2}^M\phi(n)p^2,\\
\] 
%where $\phi(\cdot)$ denotes the Euler totient function and
and
\[
	\var[X]=\sum_{\substack{i<j \\ (i,j)=1}}p^2(1-p^2)\ +\sum_{\substack{i<j<j'  \\ (i,j)=1 \text{ and } (i,j')=1}}	p^3(1-p)\le M^2p^2(1-p^2)+M^3p^3(1-p).%\sum_{n=2}^M\phi(n)p^2(1-p^2)+M^3p^3(1-p).
\]
 %By a classical result of Mertens [CITE], $\sum_{n=2}^M \phi(n)=\frac{3}{\pi^2}M^2+\bigoh{M\log M}$.
Hence, 
\[
\frac{\var[X]}{\mathbb{E}[X]^2}\le \frac{M^2p^2(1-p^2)+M^3p^3(1-p)}{M^4p^4}\to0\ \ \text{ as }M\to\infty.  
\] By the second moment method, $\prob{X>0}\to1$ and thus when $p\gg1/M$, 
$\gens$ will contain a pair of coprime integers a.a.s, which guarantees that $\nsg$ is cofinite a.a.s. in this case.
\end{proof}

%!TEX root = 0-RandomSemigroups.tex

%%%%%%%%%%%%%%%%%%%%%%%%%%%%%%%%%%%%%%%%%%%%%%%%%%%%%%%%%%%%
\section{The simplicial complex of irreducible semigroups}%%
\label{sec:complex}%%%%%%%%%%%%%%%%%%%%%%%%%%%%%%%%%%%%%%%%%
%%%%%%%%%%%%%%%%%%%%%%%%%%%%%%%%%%%%%%%%%%%%%%%%%%%%%%%%%%%%

Before proving the remaining parts of Theorem~\ref{t:maintheorem}, we introduce in Definition~\ref{d:complexofirreducibles} a simplicial complex whose combinatorial properties govern several questions arising from the ER-type model for sampling random numerical semigroups.  We prove that this complex is shellable (Proposition~\ref{p:shellable}), in the process uncovering a combinatorial interpretation of its $h$-vector entries (Corollary~\ref{c:hvector}).  We~begin by recalling the definition of a simplicial complex and some related concepts, as presented in \cite[Chapter~2]{stanleybook}.

\begin{defn}[{\cite[Chapter~2]{stanleybook}}]\label{d:simplicialcomplex}
Fix $n \in \ZZ_{\ge 1}$.  A \emph{simplicial complex} with vertices $[n] = \{1, \ldots, n\}$ is a collection $\Delta$ of subsets of $[n]$ (called \emph{faces}) such that $A \in \Delta$ implies $B \in \Delta$ whenever $B \subset A$.  The \emph{dimension} of a face $A$ is $\dim(A) = |A| - 1$ and the \emph{dimension} $d = \dim(\Delta)$ of $\Delta$ is the largest dimension among the faces of $\Delta$.  The \emph{facets} of $\Delta$ are the maximal faces with respect to~containment, and $\Delta$ is \emph{pure} if its facets all have the same dimension.  The \emph{$f$-vector} $(f_{-1}, f_0, \ldots, f_d)$ of $\Delta$ has entries giving the number
$$f_i = \#\{A \in \Delta : |A| = i+1\}$$
of $i$-dimensional faces of $\Delta$, and the \emph{$h$-vector} $(h_0, h_1, \ldots)$ of $\Delta$ has entries expressed as
$$h_i = \sum_{j = 0}^i (-1)^{i-j} \binom{d - j}{i - j}f_{j-1}$$
in terms of the $f$-vector.  
\end{defn}

\begin{defn}[{\cite[Chapter~2]{rosales+pedro}}]\label{d:irreducible}
A numerical semigroup $S$ is \emph{irreducible} if it is maximal (with respect to containment) among numerical semigroups with the same Frobenius number $F(S)$.  
\end{defn}

We now define the simplicial complex $\Delta_n$ whose facets are in natural bijection with irreducible numerical semigroups with Frobenius number $n$.

\begin{defn}\label{d:complexofirreducibles}
Fix $n \ge 2$, and let $S_1, \ldots, S_r$ denote the irreducible numerical semigroups with Frobenius number $n$.  Define $\Delta_n$ as the simplicial complex on $[n-1]$ with facets $F_i = S_i \cap [n-1]$, and let $d_n = \deg h_n(x)$, where
$$h_n(x) = h_{n,0} + h_{n,1}x + \cdots$$
denotes the polynomial in $x$ whose coefficients are the $h$-vector $(h_{n,0},h_{n,1},\ldots)$ of $\Delta_n$.  
\end{defn}

\begin{lm}\label{l:symmetric}
A numerical semigoup $S$ with Frobenius number $F(S) = n$ is irreducible if and only if $n - s \notin S$ implies $s \in S$ for $s < n/2$.  In particular, $\Delta_n$ is pure of dimension $\lfloor (n-1)/2 \rfloor$.  
\end{lm}

\begin{proof}
The first claim follows from \cite[Proposition~3.4]{rosales+pedro}, and yields a bijection between the gaps of~$S$ (excluding $n/2$) and the elements of $S$ less than $n$.  The second claim now follows.  
\end{proof}

\begin{remark}\label{r:irreduciblesemigroupgens}
Lemma~\ref{l:symmetric} implies that for an irreducible numerical semigroup $S$, the set of minimal generators less than $n/2$ (so long as it is nonempty) uniquely determines~$S$.  In particular, the minimal generators determine which integers less than $n/2$ lie in $S$, and~the fact that $i \in S$ if and only if $n-i \notin S$ determines the remainder of the gaps of $S$.  
\end{remark}

\begin{defn}[{\cite[Definition~2.1]{stanleybook}}]\label{d:shellable}
A \emph{shelling order} of a pure simplicial complex $\Delta$ is a total ordering $F_1, \ldots, F_r$ of the facets of $\Delta$ so that for every $i > 1$, the complex $\{F_i \cap F_1, \ldots, F_i \cap F_{i-1}\}$ is pure of dimension $\dim(\Delta) - 1$.  We say $\Delta$ is \emph{shellable} if it has a shelling order.  
\end{defn}

\begin{prop}\label{p:shellable}
Fix $n \ge 1$, let $S_1, \ldots, S_r$ denote the irreducible numerical semigroups with Frobenius number $n$, and let $F_i = S_i \cap [n-1]$ be the facet of $\Delta_n$ corresponding to $S_i$.  If~$\sum F_i \ge \sum F_j$ for all $i < j$, then $F_1, \ldots, F_r$ is a shelling order for $\Delta_n$.  
\end{prop}

\begin{proof}
Fix $i \ge 2$.  It suffices to prove that whenever $|F_j \cap F_i| < |F_i| - 1$ for $j < i$, there exists $k < i$ such that $F_k \cap F_i \supset F_j \cap F_i$.  Let $a = \min(F_i \setminus F_j)$ and $b = \max(F_j \setminus F_i)$.  Then $S_i \setminus \{a\}$ is closed under addition since $S_j$ is closed under addition and $F_j$ and $F_i$ have identical elements less than $a$.  Additionally, $S_i \cup \{b\} \setminus \{a\}$ is closed under addition, since its elements greater than $b$ are identical to those of $S_j$.  In particular, $S_k = S_i \cup \{b\} \setminus \{a\}$ is an irreducible numerical semigroup with Frobenius number $n$, and $k < i$ since $\sum F_k > \sum F_i$.  This completes the proof.  
\end{proof}

\begin{example}\label{e:shellingexample}
The faces in the simplicial complex $\Delta_7$ are depicted in Figure~\ref{f:shellingexample}.  The shelling order for $\Delta_7$ produced by Proposition~\ref{p:shellable} is $(456, 356, 246)$, and the unique minimal face corresponding to each facet is double-circled in the figure.  Since $h_{n,i}$ counts the number of such minimal faces with exactly $i$ elements, we have $h_7(x) = 1 + 2x$.  
\end{example}

\begin{figure}
\begin{center}
\includegraphics[width=1.0in]{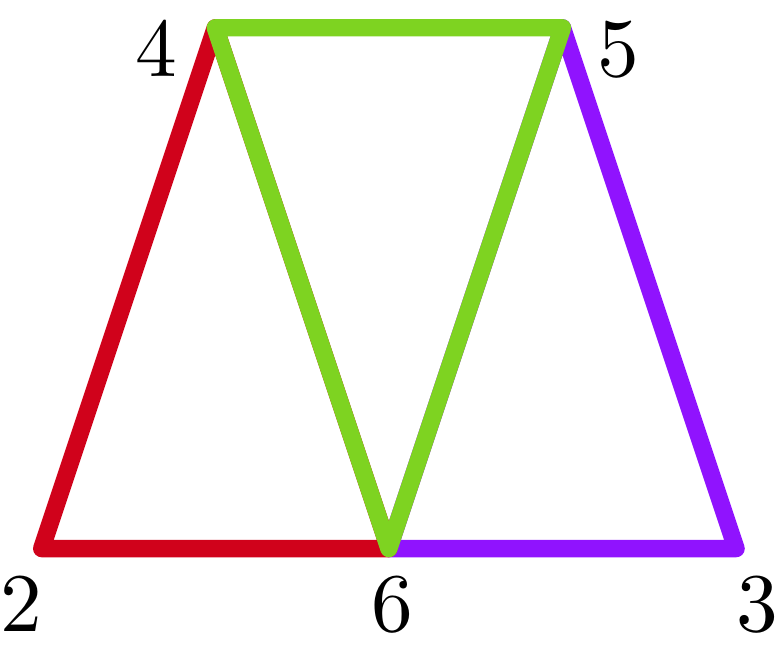}
\qquad
\includegraphics[width=2.5in]{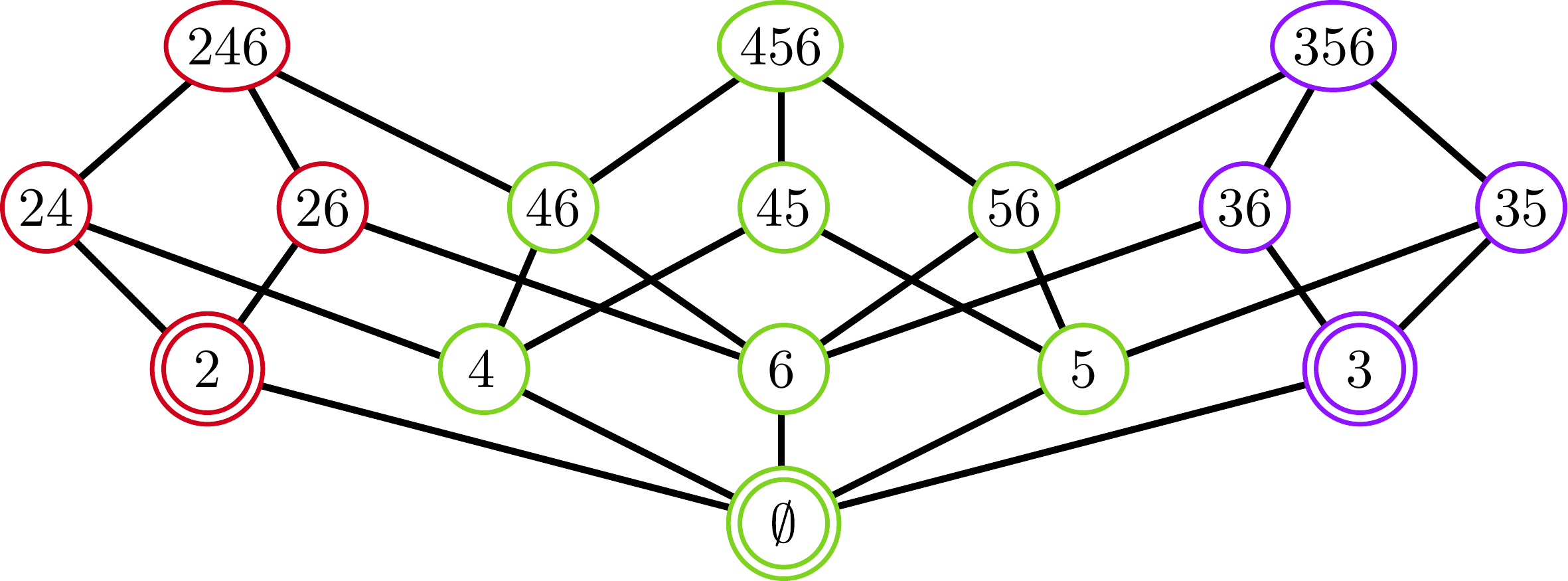}
\end{center}
\caption{The simplicial complex $\Delta_7$ (left) and its face poset (right); see Example~\ref{e:shellingexample}.}
\label{f:shellingexample}
\end{figure}

\begin{thm}\label{t:hvector}
Each $h$-vector entry $h_{n,i}$ is nonnegative and equals the number of irreducible numerical semigroups with Frobenius number $n$ with exactly $i$ minimal generators less than $n/2$.  
\end{thm}

\begin{proof}
Fix notation as in Proposition~\ref{p:shellable}.  
Any shellable simplicial complex has non-negative $h$-vector entries \cite[Corollary~3.2]{stanleybook}, so the non-negativity of each $h_{n,i}$ follows from Proposition~\ref{p:shellable}.  Moreover, under any shelling order of $\Delta_n$ (in particular, under any ordering $F_1, \ldots, F_r$ satisfying Proposition~\ref{p:shellable}), $F_j \setminus (F_1 \cup \cdots \cup F_{j-1})$ has a unique minimal face $R_j$ for each $j$, and $h_{n,i} = \#\{R_j : \#R_j = i\}$ counts the number of such minimal faces with $i$ vertices.  As such, it suffices to show each $R_i$ equals the set $G_i$ of minimal generators of $S_i$ less than $n/2$.  

Now, clearly $h_{n,0} = 1$, so assume $G_i$ is nonempty.  Any facet $F$ containing $G_i$ agrees with $F_i$ for all elements less than $n/2$, and Remark~\ref{r:irreduciblesemigroupgens} implies $F_i = F$.  As such, $F_i$ is the only facet containing the face $G_i$, and $R_i \subset G_i$.  Conversely, suppose $a \in G_i \setminus R_i$, and let $S' = S_i \setminus \{a\} \cup \{n - a\}$.  Since $a$ is a minimal generator of $S$, the set $S'$ is closed under addition and has the same number of elements less than $n$ as $S_i$, meaning $S'$ is an irreducible numerical semigroup.  Since $S_i$ contains $R_i$ and appears before $S_i$ in the shelling order, we have arrived at a contradiction.  
\end{proof}

\begin{cor}\label{c:hvector}
There is a bijection between the irreducible numerical semigroups with Frobenius number $n$ and the numerical semigroups not containing $n$ whose generators are all less than $n/2$.  In~particular, $h_{n,i}$ equals the number of embedding dimension $i$ semigroups not containing $n$ whose generators are all less than $n/2$.  
\end{cor}

\begin{proof}
For any irreducible numerical semigroup $S = \<n_1 < \cdots < n_k\>$ with $F(S) = n$, the semigroup $T = \<n_1, \ldots, n_t\>$ with $n_t < n/2 < n_{t+1}$ is contained in $S$ and thus cannot contain $n$.  Conversely, fix a numerical semigroup $T = \<n_1 < \cdots < n_t\>$ with $n_t < n/2$, let 
$$B = \{s \in (n/2,n) : s \notin T \text{ and } n - s \notin T\},$$
and let $S = T \cup B$.  If $s + s' \notin T$ for $s \in B$ and $s' \in T$, then $(n - s) - s' \notin T$, meaning $s + s' \in B$.  We conclude $S$ is closed under addition, at which point Lemma~\ref{l:symmetric} implies $S$ is irreducible.  

The second claim now follows from the first and Theorem~\ref{t:hvector}.  
\end{proof}

\begin{remark}\label{r:polycomputed}
By Theorem~\ref{t:hvector}, the coefficients $h_{n,i}$ of the polynomial $h_n(x)$ can be computed using~\cite{fundamentalgaps}, which gives an algorithm to compute the set of irreducible numerical semigroups with Frobenius number~$n$.  A~precomputed list up to $n = 90$ can be found at the following webpage (the computations for $n \ge 88$ each take over a day to complete with the authors' personal computers): 
\begin{center}
\url{https://gist.github.com/coneill-math/c2f12c94c7ee12ac7652096329417b7d}
\end{center}
The $h$-vectors of $\Delta_{89}$ and $\Delta_{90}$ (some entries of which are given in Table~\ref{tb:bounds}) demonstrate an interesting phenomenon: not only are the coefficients $h_{n,i}$ not necessarily monotone for fixed $i$, but the fewer divisors $n$ has, the larger $h_{n,i}$ tends to be with respect to the surrounding $n$-values.  This is likely due in part to the complex $\Delta_n$ having more vertices in this case.  The computations also take considerably longer in such cases; indeed, $n = 89$ took longer than for $n = 88$ and $n = 90$ combined.  

We remind the reader that the polynomial $h_n(x)$ does not depend on a given numerical semigroup; rather, there is precisely one polynomial for each $n \in \ZZ_{\ge 1}$, and it encodes information about \emph{all} numerical semigroups with Frobenius number $n$.  
% Indeed, each numerical semigroup $S$ with Frobenius number $n$ occurs as a face $S \cap [n-1] \in \Delta_n$.  
Though a wide assortment of posets whose elements are numerical semigroups have been studied elsewhere in the literature \cite{oversemigroup,frobeniusvariety}, we were unable to locate any that consider $\Delta_n$.  
\end{remark}

We now give some basic properties of the $h$-vector of $\Delta_n$.  

\begin{prop}\label{p:hvectorprops}
Fix $n \ge 1$.  
\begin{enumerate}[(a)]
\item 
$h_{n,0} = 1$.  

\item 
$h_{n,1} = \lfloor (n+1)/2 \rfloor - \tau(n)$, where $\tau(n)$ denotes the number of divisors of $n$.  

\item 
$d_n = \deg h_n(x) = \lfloor (n-1)/2 \rfloor - \lfloor n/3 \rfloor$.  

\end{enumerate}
\end{prop}

% 54 [[17,19,21,22,23,24,25,26], 
%     [19,20,21,22,23,24,25,26]]
% 55 [[18,20,21,22,23,24,25,26,27], 
%     [19,20,21,22,23,24,25,26,27]]
% 56 [[19,20,21,22,23,24,25,26,27]]
% 57 [[18,20,22,23,24,25,26,27,28], 
%     [20,21,22,23,24,25,26,27,28]]
% 58 [[19,21,22,23,24,25,26,27,28], 
%     [20,21,22,23,24,25,26,27,28]]
% 59 [[20,21,22,23,24,25,26,27,28,29]]

\begin{proof}
We proceed using Corollary~\ref{c:hvector} and characterizing the possible sets $A \subset (0,n/2)$ of integers minimally generating a numerical semigroup $S = \<A\>$ with $n \notin S$.  
Since $A = \emptyset$ generates the semigroup $S = \{0\}$, we have $h_{n,0} = 1$.  Additionally, any non-divisor of $n$ less than $n/2$ generates a semigroup not containing $n$, which proves part~(b).  

Now, Theorem~\ref{t:bounds} implies $d_n \le \lfloor (n-1)/2 \rfloor - \lfloor n/3 \rfloor$.  Moreover, the set $A = (n/3,n/2) \cap \ZZ$ minimally generates a semigroup not containing $n$ since the sum of any two elements is strictly less than~$n$ while the sum of any three is strictly larger than $n$.  Since $|A| = d_n$, this proves~(c).  
\end{proof}

We conclude this section with the following bounds on the $h$-vector entries of $\Delta_n$, which play a crutial role in establishing the final threshold function in Section~\ref{sec:mingens} and estimating several expected values in Section~\ref{sec:approx}.  

\begin{thm}\label{t:bounds}
For any $n \ge 1$, $i \ge 1$, and $N \ge 2$, we have
$$\sum_{j = 2}^N \binom{\lfloor (n-1)/j \rfloor - \lfloor n/(j+1) \rfloor}{i} \le h_{n,i} \le \binom{\lceil n/2 \rceil - 2i}{i}.$$
% In particular, for $i \ge 1$ fixed and $n \gg 0$, $h_i \sim n^i$.  
\end{thm}

\begin{proof}
By Corollary~\ref{c:hvector}, $h_{n,i}$ counts sets $A$ minimally generating a semigroup not containing $n$.  We claim $m = \min(A) \ge 2i$.  Indeed, since $A$ forms a minimal generating set, each element must be distinct modulo $m$.  Additionally, since $n \notin \langle A \rangle$, any element $a \in A$ cannot satisfy $a \equiv n \bmod m$, and since $\max(A) \le \lfloor n/2 \rfloor$, any two elements $a, b \in A$ cannot satisfy $a + b \equiv n \bmod m$.  
The upper bound immediately follows.  

For the lower bound, we claim that for any $j \ge 2$, each set $A$ of $i$ distinct integers chosen from the open interval $(n/(j+1), n/j)$ minimally generates a numerical semigroup $S$ with $n \notin S$.  Indeed,~the sum of any $j$ elements of $A$ is strictly less than $n$, while the sum of any $j+1$ is strictly larger than~$n$.  Additionally, since $j \ge 2$, the sum of any two elements of $A$ exceeds $n/j$, ensuring $A$ minimally generates $S$.  This completes the proof.  
\end{proof}

\begin{remark}\label{r:bounds}
Proposition~\ref{p:hvectorprops} implies the lower bound in Theorem~\ref{t:bounds} is tight for $i = 1$.  The~given bounds are tighter than those for more general Cohen-Macaulay simplicial complexes \cite{stanleybook} and are sufficient to prove the results in the coming sections, but still leave room for improvement.  Table~\ref{tb:bounds} compares values from Theorem~\ref{t:bounds} with those computed in Remark~\ref{r:polycomputed}.  
\end{remark}

\begin{table}[tbp]
\begin{center}
\begin{tabular}{|l|l|l|l||l|l|l|l||l|l|l|}
\hline
\multicolumn{4}{|l||}{$n = 89$} & 
\multicolumn{4}{l||}{$n = 90$} & 
\multicolumn{3}{l|}{$n = 500$} \\
\hline
$i$ & Lower & Actual & Upper & $i$ & Lower & Actual & Upper & $i$ & Lower & Upper \\
\hline
1  & 43   & 43    & 43      & 1  & 34   & 34    & 43      & 10 & $2.4 \cdot 10^{12}$ & $9.3 \cdot 10^{16}$ \\
3  & 501  & 3873  & 9139    & 3  & 403  & 2442  & 9139    & 20 & $8.1 \cdot 10^{18}$ & $4.4 \cdot 10^{27}$ \\
5  & 3025 & 27570 & 324632  & 5  & 2023 & 16065 & 324632  & 30 & $3.4 \cdot 10^{22}$ & $7.7 \cdot 10^{34}$ \\
7  & 6436 & 39358 & 2629575 & 7  & 3433 & 21213 & 2629575 & 40 & $8.0 \cdot 10^{23}$ & $1.3 \cdot 10^{39}$ \\
9  & 5005 & 18186 & 4686825 & 9  & 2002 & 8343  & 4686825 & 50 & $1.4 \cdot 10^{23}$ & $2.0 \cdot 10^{40}$ \\
11 & 1365 & 3044  & 1352078 & 11 & 364  & 1055  & 1352078 & 60 & $1.8 \cdot 10^{20}$ & $6.4 \cdot 10^{37}$ \\
13 & 105  & 153   & 27132   & 13 & 14   & 31    & 27132   & 70 & $5.2 \cdot 10^{14}$ & $1.6 \cdot 10^{30}$ \\
\hline
\end{tabular}
\end{center}
\caption[Comparison of bounds in Theorem~\ref{t:bounds}]{Comparison of the bounds in Theorem~\ref{t:bounds} with values computed in Remark~\ref{r:polycomputed}.}
\label{tb:bounds}
\end{table}
% N = 500
% [sum([binomial(floor((N-1)/j) - floor(N/(j+1)),i) for j in [2..N] if floor((N-1)/j) - floor(N/(j+1)) >= i]) for i in [10,20..70]]
% [binomial(ceil(N/2)-2*i,i) for i in [10,20..70]]

% \item 
% For $n > 7$, the leading coefficient of $h_n(x)$ is 1 if $n \equiv 2 \bmod 3$ and 2 otherwise.  

% \chris{characterize maximal generating sets counted by $h$-vector.  Use cases based on $n \bmod 6$}

% \chris{Do we want to include the following result?  It kind of detracts from the complex of irreducibles.}

% \begin{prop}
% The generating sets counted in the preceeding Proposition form a simpicial complex.  \chris{complex of generating sets?}  In particular, the coefficients of $h_n(x)$ equal the $f$-vector of some simplicial complex $C_n$.  Moreover, for $n > 12$, the only possible isolated vertices of $C_n$ are 
% \begin{itemize}
% \item 
% $\{2\}$, which occurs when $2 \nmid n$, 

% \item 
% $\{3\}$, which occurs when $3 \nmid n$, and

% \item 
% $\{4\}$, which occurs when $n \equiv 2 \bmod 4$.  

% \end{itemize}
% \end{prop}

% \begin{proof}
% \chris{Omitting generators doesn't cause issues, and the rest shouldn't be hard to trace through in cases}
% \end{proof}

%!TEX root = 0-RandomSemigroups.tex

%%%%%%%%%%%%%%%%%%%%%%%%%%%%%%%%%%%%%%%%%%%%%%%%%%%%%%%%%%%%
\section{Expected number of minimal generators}%%%%%%%%%%%%%
\label{sec:mingens}%%%%%%%%%%%%%%%%%%%%%%%%%%%%%%%%%%%%%%%%%
%%%%%%%%%%%%%%%%%%%%%%%%%%%%%%%%%%%%%%%%%%%%%%%%%%%%%%%%%%%%

The main result of this section is Corollary~\ref{c:embdim}, which states that if $p \to 0$ as $M \to \infty$, then the expected number of generators, expected number of gaps, and expected Frobenius number are all unbounded.  Our proof uses a surprising connection between the probability $a_n(p)$ that a non-negative integer $n$ lies in the chosen semigroup (Definition~\ref{d:anp}) and the $h$-vector of the simplicial complex $\Delta_n$ introduced in Section~\ref{sec:complex}; see Remark~\ref{r:hilbertseries}.  

\begin{conv}\label{conv:cofinitesubtlety}
In the remainder of the paper, we adhere to the convention that if $S$ is a numerical semigroup that is not cofinite, then $g(S) = F(S) = 0$.  We do this so as not to affect $\expect{g(\nsg)}$ and $\expect{F(\nsg)}$; indeed, the assumption $1/M \ll p(M)$ is made everywhere either of those two quantities appears, so Theorem~\ref{t:cofinite} ensures that $S$ is cofinite a.a.s.  
\end{conv}

\begin{defn}\label{d:anp}
Let $\nsg \sim S(M,p)$.  For each integer $n \in [M]$, denote by $a_n(p)$ the probability that $n$ \textit{cannot} be written as a sum of elements in $\nsg\cap[n-1]$, that is,
$$a_n(p) := \prob{n \notin \langle \nsg \cap [n-1] \rangle}.$$
\end{defn}

\begin{prop}\label{p:anp}
For each $n \ge 1$, $a_n(p) = (1 - p)^{\lfloor n/2 \rfloor} h_n(p)$.
\end{prop}

\begin{proof}
The faces of $\Delta_n$ are precisely the sets $A \subset [n-1]$ satisfying $n \notin \langle A \rangle$.  As such, 
$$a_n(p) = \sum_{A \in \Delta_n} p^{|A|}(1 - p)^{n-1-|A|} = (1 - p)^{\lfloor n/2 \rfloor}\sum_{i = 0}^{d_n} f_ip^{i}(1 - p)^{d_n-i} = (1 - p)^{\lfloor n/2 \rfloor} h_n(p)$$
follows upon unraveling Definitions~\ref{d:simplicialcomplex} and~\ref{d:complexofirreducibles}.  
\end{proof}

\begin{remark}\label{r:hilbertseries}
By Proposition~\ref{p:anp}, we can write
$$a_n(p) = (1-p)^n \mathcal H(\kk[x]/I_{\Delta_n};p)$$
in terms of the Hilbert series of the \emph{Stanley-Reisner ring} $\kk[x]/I_{\Delta_n}$ (see \cite[Chapter~2, Definition~1.1]{stanleybook}).  We encourage the interested reader to consult~\cite{stanleybook} for background on Hilbert functions.  
\end{remark}

\begin{thm}\label{thm:embdim}
Let $\nsg \sim S(M,p)$, where $p = p(M)$ is a monotone decreasing function of $M$.  If~$1/M \ll p \ll 1$, then $\lim_{M \to \infty} \expect{e(\nsg)} = \infty$.
% If $\lim_{M \to \infty} p(M) > 0$, then $\lim_{M \to \infty} \expect{e(\nsg)} < \infty$.
\end{thm}

\begin{proof}
By Proposition~\ref{p:anp} and the linearity of expectation,
\[
\expect{e(\nsg)}
= \sum_{n=1}^M pa_n(p)
= \sum_{n=1}^M p(1-p)^{\lfloor n/2 \rfloor}h_n(p).
\]
For any fixed integer $N \ge 2$, the lower bound in Theorem~\ref{t:bounds} gives
\begin{align*}
\expect{e(\nsg)}
&\ge \sum_{n=1}^M p(1-p)^{\lfloor n/2 \rfloor} \sum_{i=0}^{d_n} \sum_{j = 2}^N \binom{\lfloor (n-1)/j \rfloor - \lfloor n/(j+1) \rfloor}{i}p^i \\
&= \sum_{j = 2}^N \sum_{n=1}^M p(1-p)^{\lfloor n/2 \rfloor} \sum_{i=0}^{d_n} \binom{\lfloor (n-1)/j \rfloor - \lfloor n/(j+1) \rfloor}{i}p^i.
\end{align*}
We now consider each summand $\sigma_j$ of the outer sum for a fixed value of $j$.  Using the division algorithm to write each $n = k (2j(j+1)) + r$ for $k \ge 0$ and $1 \le r \le 2j(j+1)$, and supposing $M = m (2j(j+1))$ for some $m \in \ZZ_{\ge 1}$, we obtain
\begin{align*}
\sigma_j
&= \sum_{n=1}^M p(1-p)^{\lfloor n/2 \rfloor} \sum_{i=0}^{d_n} \binom{\lfloor (n-1)/j \rfloor - \lfloor n/(j+1) \rfloor}{i}p^i \\
&= \sum_{n=1}^M p(1-p)^{\lfloor n/2 \rfloor}(1+p)^{\lfloor (n-1)/j \rfloor - \lfloor n/(j+1) \rfloor} \\
&= \sum_{k=1}^m p(1-p)^{kj(j+1)+\bigoh{1}}\sum_{r=1}^{2j(j+1)} (1+p)^{2k(j+1)-2kj+\bigoh{1}} \\
&= \left(\sum_{r=1}^{2j(j+1)} (1-p)^{\bigoh{1}}(1+p)^{\bigoh{1}}\right) \sum_{k=1}^m p\left((1-p)^{j(j+1)}(1+p)^2\right)^{k-1} \\
&= \left(\sum_{r=1}^{2j(j+1)} (1-p)^{\bigoh{1}}(1+p)^{\bigoh{1}}\right) p\frac{1 - \left((1-p)^{j(j+1)}(1+p)^2\right)^m}{1 - (1-p)^{j(j+1)}(1+p)^2}.
\end{align*}
Since $p \gg 1/M$, a simple calculus exercise shows $\left((1-p)^{j(j+1)}(1+p)^2\right)^m \to 0$.  If $p \to 0$, we obtain
$$\expect{e(\nsg)} \ge \sum_{j=2}^N \sigma_j \to \sum_{j=2}^N \frac{2j(j+1)}{j(j+1)-2} \ge 2(N-1),$$
which must hold for every $N \ge 2$.  
\end{proof}

\begin{cor}\label{c:embdim}
Resuming notation from Theorem~\ref{thm:embdim}, if $1/M \ll p \ll 1$, then 
$$\lim_{M \to \infty} \expect{e(\nsg)} = \lim_{M \to \infty} \expect{g(\nsg)} = \lim_{M \to \infty} \expect{F(\nsg)} = \infty.$$
\end{cor}

\begin{proof}
Apply Theorem~\ref{thm:embdim} and the inequalities
$$e(S) - 1 \le \min(S \setminus \{0\}) - 1 \le g(S) \le F(S),$$
which hold for any cofinite numerical semigroup $S$.  
\end{proof}

%!TEX root = 0-RandomSemigroups.tex

%%%%%%%%%%%%%%%%%%%%%%%%%%%%%%%%%%%%%%%%%%%%%%%%%%%%%%%%%%%%
\section{Approximations}%%%%%%%%%%%%%%%%%%%%%%%%%%%%%%%%%%%%
\label{sec:approx}%%%%%%%%%%%%%%%%%%%%%%%%%%%%%%%%%%%%%%%%%%
%%%%%%%%%%%%%%%%%%%%%%%%%%%%%%%%%%%%%%%%%%%%%%%%%%%%%%%%%%%%

In the final section of this paper, we prove the only remaining case in Theorem~\ref{t:maintheorem}, namely where $p$ is bounded away from zero (Theorem~\ref{t:fixedpbounds}).  In this case, it suffices to assume $p \in (0,1)$ is constant.  In doing so, we provide explicit bounds on $\expect{e(\nsg)}$, $\expect{g(\nsg)}$, and $\expect{F(\nsg)}$ as $M \to \infty$ using the $h$-vector bounds in Theorem~\ref{t:bounds}; Remark~\ref{r:fixedpbounds} discusses the accuracy of these estimates.  

\begin{lm}\label{l:gapcount}
For any cofinite numerical semigroup $S$ and $M \le F(S)$, $|S \cap [0,M]| \le M/2$.
\end{lm}

\begin{proof}
The key observation is that if $a \in S$, then $F(S) - a \notin S$, so at least half of the integers less than $F(S)$ lie outside of $S$.  As such, if $|S \cap [0,M]| > M/2$ and $F$ is the smallest gap of $S$ not less than $M$, the semigroup $T = S \cup [F+1,\infty)$ violates the observation.  
\end{proof}

Recall that $g_M(S)$ denotes the number of gaps $n$ of $S$ such that $n \le M$.  

\begin{lm}\label{l:gapsbelowm}
Let $\nsg \sim S(M,p)$ where $p \in(0,1)$ is consant.  Then,
\[
\lim_{M \to \infty} \expect{g(\nsg)} = \lim_{M \to \infty} \expect{g_M(\nsg)}.
\]
\end{lm}
\begin{proof}
We split the expectation $\expect{g(\nsg)}$ into two parts:
\[
\expect{g(\nsg)}
= \sum_{n=1}^{\infty} \prob{n \in G(\nsg)}
= \sum_{n=1}^M \prob{n \in G(\nsg)} + \sum_{n>M} \prob{n \in G(\nsg)}.
\]
%+\sum_{n=M+1}^{M^2} \prob{n\in G(\nsg)}+\sum_{n>M^2} \prob{n\in G(\nsg)}.
%\]
The first term is just $\expect{g_M(\nsg)}$.  For the second term, notice that if $n > M$ is a gap, then no consecutive pair of integers $m, m+1$ in the range $[1,\sqrt{M}]$ can be in $\nsg$, since 
$$F(\langle m, m+1 \rangle) = m(m+1) - 2m - 1 \le M.$$
For each even integer $m \le \sqrt{M}$, let $E_{m}$ denote the event that $m$ and $m-1$ are not both in $\gens$.  Note that $\prob{E_m} = 1-p^2$ and that for $m \ne m'$, the events $E_m$ and $E_{m'}$ are independent.  Hence,
\[
\prob{n \in G(S)}
\le \mathbb{P} \left[E_2 \cap E_4 \cap \cdots\right]
% = \prod_{\substack{ m\le\sqrt{M}-1 \\ m\text{ odd} }}1-p^2
= (1-p^2)^{\left\lfloor \frac{1}{2}\sqrt{M}\right\rfloor},
% \to 0 \text{ as } M \to \infty. 
\]
which implies the desired result.

%By Theorem~\ref{t:cofinite}, the third term vanishes as $M\to\infty$ for fixed $p$.  For the second term, $n \in G(\nsg)$ implies $g_M(\nsg)\ge M/2$ by Lemma~\ref{l:gapcount}.  Hence,
%\[
%\prob{n \in G(\nsg)} \le \prob{g_M(\nsg) \ge M/2}\le \prob{|\gens| \le M/2} = \littleoh{e^{-M^2}},
%\]
%where the final equality uses Hoeffding's inequality \cite[Theorem~TODO]{AlonSpencer} since $|\gens| \sim \text{Binomial}(M,p)$.  
\end{proof}

\begin{prop}\label{p:fixedpbounds}
For constant $p \in(0,1)$, 
$$\frac{6 - 8p + 3p^2}{2p - 2p^3 + p^4} \le \lim_{M \to \infty} \sum_{n=1}^M a_n(p) \le \frac{2 - p^2}{p^2}.$$
\end{prop}

\begin{proof}
We begin with the upper bound.  Suppose $M = 2m$ for some $m \in \ZZ_{\ge 1}$.  By Theorem~\ref{t:bounds}, 
\begin{align*}
\sum_{n=1}^M a_n(p) 
&\le \sum_{n=1}^M (1-p)^{\lfloor n/2 \rfloor} \sum_{i = 0}^{d_n} \binom{\lfloor n/2 \rfloor}{i} p^i
\le \sum_{n=1}^M (1-p)^{\lfloor n/2 \rfloor} \sum_{i = 0}^{\lfloor n/2 \rfloor} \binom{\lfloor n/2 \rfloor}{i} p^i \\
&= \sum_{n=1}^M (1-p^2)^{\lfloor n/2 \rfloor}
= 1 + 2\sum_{n=1}^m (1-p^2)^n
= 1 + 2(1 - p^2)\frac{1 - (1-p^2)^m}{1 - (1 - p^2)},
\end{align*}
which yields the claimed upper bound since $(1-p^2)^m \to 0$ as $m \to \infty$.  For the lower bound, writing $n = 6k + r$ for $r \in [6]$ and letting $M = 6m$ for some $m \in \ZZ_{\ge 1}$, we use Theorem~\ref{t:bounds} to obtain
\begin{align*}
\sum_{n=1}^M a_n(p) 
&\ge \sum_{n=1}^M (1-p)^{\lfloor n/2 \rfloor} \sum_{i = 0}^{d_n} \binom{d_n}{i} p^i
= \sum_{n=1}^M (1-p)^{\lfloor n/2 \rfloor} (1+p)^{d_n} \\
% &= \sum_{r = 1}^6 \sum_{k = 0}^m (1-p)^{3k + \lfloor r/2 \rfloor} (1+p)^{k + \lfloor (r-1)/2 \rfloor - \lfloor r/3 \rfloor} \\
% &= \left(\sum_{r = 1}^6 (1-p)^{\lfloor r/2 \rfloor} (1+p)^{\lfloor (r-1)/2 \rfloor - \lfloor r/3 \rfloor} \right)\sum_{k = 0}^m ((1-p)^3 (1+p))^k,
&= (6 - 8p + 3p^2)\sum_{k = 0}^m ((1-p)^3 (1+p))^k,
% &= (6 - 8p + 3p^2)\frac{1 - \left((1-p)^3 (1+p)\right)^{m+1}}{1 - (1-p)^3 (1+p)},
\end{align*}
and since $\left((1-p)^3 (1+p)\right)^{m+1} \to 0$ as $m \to \infty$, the desired lower bound is obtained.  
\end{proof}

\begin{thm}\label{t:fixedpbounds}
Let $\nsg \sim S(M,p)$ where $p \in(0,1)$ is constant.  Then
\begin{align*}
\frac{6 - 8p + 3p^2}{2 - 2p^2 + p^3}
% (6 - 8*p + 3*p^2)/(2 - 2*p^2 + p^3)
&\le \lim_{M \to \infty} \expect{e(\nsg)}
\le \frac{2 - p^2}{p}, \\
\frac{6 - 14p + 11p^2 - 3p^3}{2p - 2p^3 + p^4}
&\le \lim_{M \to \infty} \expect{g(\nsg)}
\le \frac{(1-p)(2 - p^2)}{p^2}, \text{ and } \\
\frac{6 - 14p + 11p^2 - 3p^3}{2p - 2p^3 + p^4}
&\le \lim_{M \to \infty} \expect{F(\nsg)}
\le \frac{2(1-p)(2 - p^2)}{p^2}.
\end{align*}
\end{thm}

\begin{proof}
Lemma~\ref{l:gapcount} implies $g(S) \le F(S) \le 2g(S)$ for any cofinite numerical semigroup $S$.  Since 
$$\expect{e(\nsg)} = p\sum_{n = 1}^M a_n(p) \qquad \text{ and } \qquad \expect{g_M(\nsg)} = (1-p)\sum_{n = 1}^M a_n(p),$$
each claimed inequality follows from Lemma ~\ref{l:gapsbelowm} and Proposition~\ref{p:fixedpbounds}.  
\end{proof}

\begin{remark}\label{r:fixedpbounds}
Neither of the bounds in Proposition~\ref{p:fixedpbounds} use the full strength of Theorem~\ref{t:bounds}.  The~summands for $j \ge 3$ in the lower bound yield increasingly complicated rational functions in~$p$, though the resulting sequence of values necessarily converges to 0 by Theorem~\ref{t:fixedpbounds}.  Additionally, substituting the upper bound for $h_{n,i}$ given in Theorem~\ref{t:bounds} yields a sum that is nontrivial to unravel in a way sufficient to compute the limit as $M \to \infty$, and doing so would likely only marginally improve the resulting upper bounds in Theorem~\ref{t:fixedpbounds}.  Indeed, it is the larger values of $i$ whose terms benefit from the improved upper bound, and these terms are rendered negligible by the large exponent in the accompanying value $p^i$.  

We see here the need for improvements in the upper bound in Theorem~\ref{t:bounds}.  Several computed values can be found in Table~\ref{tb:fixedpbounds}.  For instance, the upper bound for $\expect{e(\nsg)}$ in Theorem~\ref{t:fixedpbounds} could be made better by simply noting that $e(S)$ is at most the smallest generator (which has expected value $1/p$).  Such improvements to Theorem~\ref{t:bounds} should be possible, given the precise characterization of the $h$-vector of $\Delta_n$ in Corollary~\ref{c:hvector}.  

It is also worth noticing that the polynomials computed in Remark~\ref{r:polycomputed} are not sufficient for accuracy for the $p$ values in Table~\ref{tb:fixedpbounds}.  Indeed, with $p = 0.01$ and $p = 0.001$, each partial sum for $M = 90$ fails to reach the lower bound, and even when $p = 0.1$, the last 10 summands (i.e.\ for $n = 81, \ldots, 90$) each lie between 0.01 and 0.02, so the next several summands will likely still contribute significantly to the limit.  
\end{remark}

\begin{table}[tbp]
\begin{center}
\begin{tabular}{|l|l|l|l|l|l|l|}
\hline
& \multicolumn{1}{c|}{Lower} & \multicolumn{2}{c|}{Experiments} & \multicolumn{1}{c|}{Upper} & Rmk~\ref{r:polycomputed} \\
% \hline
$p$ & Thm~\ref{t:fixedpbounds} & $M = 25000$ & $M = 50000$ & Thm~\ref{t:fixedpbounds} & $M = 90$ \\
\hline
0.25  & 2.21  & 3.3663  & 3.3761  & 3.75   & 2.767 \\
0.1   & 2.64  & 4.6236  & 4.6402  & 19.9   & 3.782 \\
0.01  & 2.96  & 9.7906  & 9.776   & 199.9  & 0.858 \\
0.001 & 2.996 & 15.3096 & 16.9539 & 1999.9 & 0.089 \\
\hline
\end{tabular}
\end{center}
\caption[Comparison of bounds in Theorem~\ref{t:fixedpbounds}]{Comparing estimates of $\expect{e(\nsg)}$ using the bounds in Theorem~\ref{t:fixedpbounds}, exact computation using the polynomials in Remark~\ref{r:polycomputed}, and experimental evidence from 100,000 samples.}
\label{tb:fixedpbounds}
\end{table}

% p = 1/10
% sum([p*(1-p)**floor(N/2)*sum([H[N][i]*p**i for i in range(floor((N-1)/2) - floor(N/3))]) for N in H])*1.0

\begin{proof}[Proof of Theorem~\ref{t:maintheorem}]
Apply Theorems~\ref{t:cofinite}, Corollary~\ref{c:embdim} and Theorem~\ref{t:fixedpbounds}.  
\end{proof}

\excise{
Here we present a whole bunch of conjectures and further work to be done.

\begin{enumerate}
\item Prove/disprove Wilf's conjecture probabilistically
\subitem $\bullet$ Idea: Construct $S$ with certain value of $e(S)$ in such a way that we can remove things from the left without changing conductor
\subitem $\bullet$ Expected Wilf number
\item Prove many other things about semigroups probabilistically (think about 2-colorings of [0,M], etc.)
\item QUESTION: Do a random $k$-coloring of the integers $[1,M].$  How large must $k$ be to ensure that the semigroup generated by one of the color classes is symmetric?
\item Look at survey paper to see if there are other problems to attack
\end{enumerate}

}

%    Bibliographies can be prepared with BibTeX using amsplain,
%    amsalpha, or (for "historical" overviews) natbib style.
\bibliographystyle{amsplain}
%    Insert the bibliography data here.
\bibliography{RNS}

\end{document}